\documentclass[preprint,11pt]{elsarticle}
\usepackage{lineno,hyperref}
\modulolinenumbers[5]

\usepackage{color}
\usepackage{amsmath}
\usepackage{amsfonts,amsthm,amssymb}
\usepackage{amsthm}
\newtheorem{remark}{Remark}
\usepackage{amsfonts}
\usepackage{graphics}
\usepackage{pstricks}
\usepackage{graphicx}
\usepackage{cleveref}
\usepackage{extarrows,chngpage,array,float}
\usepackage{amssymb}
\usepackage{latexsym,bm}
\usepackage{tikz-cd}
\usepackage{natbib}
\newtheorem{theorem}{Theorem}[section]
\newtheorem{lemma}[theorem]{Lemma}
\newtheorem{definition}[theorem]{Definition}
\newtheorem{proposition}[theorem]{Proposition}
\newtheorem{corollary}[theorem]{Corollary}

\newtheorem{problem}[theorem]{Problem}

\newtheorem{example}[theorem]{Example}
\allowdisplaybreaks

\journal{arXiv}

\begin{document}
	\begin{frontmatter}
		
		\title{Orbits and self-twuality in set systems and delta-matroids}

		\author{Zhuo Li$^{1}$, \ \ Xian'an Jin$^{1}$, Qi Yan$^{2}$\footnote{Corresponding author.}\\
			\small $^{1}$School of Mathematical Sciences, Xiamen University, P. R. China\\
			\small $^{2}$School of Mathematics and Statistics, Lanzhou University, P. R. China\\
			\small Email: lzhuo@stu.xmu.edu.cn, xajin@xmu.edu.cn, yanq@lzu.edu.cn}
		
		\begin{abstract}
			
			We introduce a new group action on set systems, constructed as a semidirect product of a permutation group and a group generated by twist and loop complementation operations on a single element. This action extends the ribbon group framework of Abrams and Ellis-Monaghan from ribbon graphs to set systems, facilitating a systematic investigation of self-twuality. We prove that different forms of self-twuality propagate through orbits under the group action and establish a characterization of the orbit of a vf-safe delta-matroid via multimatroids. As an application, we analyze orbits of ribbon-graphic delta-matroids. Our work answers a question posed by Abrams and Ellis-Monaghan and provides a unified algebraic framework for studying self-twuality in combinatorial structures.
		\end{abstract}
		
		\begin{keyword}
			delta-matroid, multimatroid, orbit, group action, ribbon graph
		\end{keyword}
	\end{frontmatter}

	\section{Introduction}
	Delta-matroids serve as a natural generalization of embedded graphs, much as matroids abstract and extend key structural properties of graphs. This analogy between graph theory and matroid theory finds a parallel in the relationship between embedded graphs and delta-matroids, and a significant portion of research on delta-matroids originates from generalizations of results about embedded graphs. For instance, Chun, Moffatt, Noble, and Rueckriemen \cite{Chun2019} showed that delta-matroids arise as the natural extension of graphic matroids to the setting of embedded graphs, establishing that fundamental ribbon graph operations and concepts admit natural counterparts in the delta-matroid setting. For additional related work, we refer readers to \cite{Bouchet1989}, \cite{Chun2019F}, \cite{Moffatt2017}, and \cite{Yan2024}.
	
	Wilson \cite{Wilson1979} showed that geometric duality and Petrie duality, being non-commuting involutions, generate an action of the symmetric group \(S_3\) on regular maps. Subsequently, geometric and Petrie dualities were refined to operate on individual edges of embedded graphs, beginning with Chmutov's partial duality \cite{Chmutov2019}, and followed by the twisted duals introduced in \cite{Ellis2012}, which apply all six operations of the Wilson group to individual edges.
	
	Recently, Abrams and Ellis-Monaghan \cite{Abrams2022} defined a new ribbon group action on ribbon graphs using a semidirect product of a permutation group and the original ribbon group of Ellis-Monaghan and Moffatt. They established that the orbits under this group action provide a natural framework for analyzing self-twuality, thus posing a fundamental question:
	
	\begin{problem}
		There are recently developed frameworks extending twuality operations of partial duality and partial Petrie duals to delta-matroids, and the ribbon group action lifted to this setting would open investigation into various forms of self-twuality for delta-matroids.
	\end{problem}
	
	In fact, partial duality and twists as well as partial Petrie duals and loop complementations are compatible \cite{Chun2019F}. We consider the operations of twist and loop complementation on a single element of a set system. These two operations generate a group \( \mathcal{G} \) isomorphic to the symmetric group \( S_3 \). We define a group action of \( \mathcal{G}^n \rtimes_\phi S_n \) on set systems with ground set \( [n] \), where the semidirect product is taken with respect to a homomorphism \( \phi: S_n \to \operatorname{Aut}(\mathcal{G}^n) \). A set system \( D \) is said to be \emph{self-twuality} via \( \pi \) if there exists \( (\hat{g}, \pi) \in \mathcal{G}^n \rtimes_\phi S_n \) such that \( (\hat{g}, \pi) D = D \).
	
	We introduce various forms of set system twuality and present an algebraic framework for determining these self-twualities. We then show that the twuality properties of a set system propagate through its orbits, and thus can be studied via the orbits under this group action. Hence, we intend to investigate the orbit of a delta-matroid under this group action. Notably, Ellis-Monaghan and Moffatt \cite{Ellis2012} established an elegant result for embedded graphs:
	\[
	\operatorname{Orb}_\iota(G) = \{ H \mid H_m \cong G_m \},
	\]
	where \( G_m \) denotes the medial graph of an embedded graph \( G \). However, since delta-matroids are not closed under the group action \( \mathcal{G} \), we restrict our attention to \emph{vf-safe} delta-matroids, which form a class closed under this action. For a vf-safe delta-matroid \( D \), we characterize its orbit in terms of a multimatroid construction as follows:
	\[
	\operatorname{Orb}(D) = \left\{ D' \mid \exists \tau' \in \mathbb{T}_{(U,\Omega)},\ \sigma' \in \mathbb{P}_{(U,\Omega)} \text{ such that } Z_{D', \tau', \sigma'} = Z_{D, \tau, \sigma} \right\},
	\]
	where \( \mathbb{T}_{(U,\Omega)} \) and \( \mathbb{P}_{(U,\Omega)} \) denote the sets of all transversal 3-tuples and projections of \( (U, \Omega) \), respectively.
	
	The structure of this paper is as follows. Section~\ref{2} presents the basic theory of delta-matroids, twists, and loop complementations. In Section~\ref{3}, we introduce our main algebraic framework, extending the algebraic theory of ribbon graphs developed in \cite{Abrams2022}.
	Within this framework, the study of self-twuality corresponds to analyzing stabilizers under the group action. In Section~\ref{4}, we establish that different forms of self-twuality propagate along orbits under the group action, a conclusion presented in Theorems~\ref{near conjugate} and \ref{Orb}. Section~\ref{5} provides a characterization of the orbits of vf-safe delta-matroids via multimatroids. Finally, Section~\ref{6} applies a degenerate case of Theorem~\ref{main} to investigate the orbits of ribbon-graphic delta-matroids.
	
	\section{Delta-matroids and fundamental operations}\label{2}
	
	In this section, we review several definitions related to delta-matroids and refer the reader to \cite{Bouchet1987} and \cite{Chun2019} for further details.
	
	A \emph{set system} is a pair \( D = (E, \mathcal{F}) \), where \( E \) is a finite set, called the \emph{ground set}, and \( \mathcal{F} \) is a collection of subsets of \( E \), called \emph{feasible sets}. In this paper, it will be convenient to identify \( E \) with the set \( [n] = \{1, 2, \dots, n\} \). A set system \( D \) is said to be \emph{proper} if \( \mathcal{F} \neq \emptyset \), and \emph{normal} if the empty set is feasible. For notational convenience, we omit braces for singletons \( \{i\} \). Bouchet \cite{Bouchet1987} introduced delta-matroids as follows.
	
	\begin{definition}
		\normalfont
		A \emph{delta-matroid} is a proper set system \( D = ([n], \mathcal{F}) \) that satisfies the \emph{Symmetric Exchange Axiom}:
		for any \( X, Y \in \mathcal{F} \) and any \( u \in X \triangle Y \), there exists a \( v \in X \triangle Y \) (possibly \( v = u \)) such that \( X \triangle \{u, v\} \in \mathcal{F} \). Here \( X \triangle Y = (X \cup Y) \setminus (X \cap Y) \) denotes the symmetric difference of sets.
	\end{definition}
	
	We consider two operations on set systems: twisting and loop complementation.
	Twisting was introduced by Bouchet~\cite{Bouchet1987}, and loop complementation was introduced by Brijder and Hoogeboom~\cite{Brijder2011}.
	
	For a set system \( D = ([n], \mathcal{F}) \) and any \( I \subseteq [n] \), let
	\[
	\mathcal{F} \ast I = \{ X \triangle I : X \in \mathcal{F} \},
	\]
	and define the \emph{twist} of \( D \) with respect to \( I \), denoted by \( D \ast I \), as \( ([n], \mathcal{F} \ast I) \).
	It is straightforward to show that the twist of a delta-matroid is also a delta-matroid.
	In particular, \( D \ast [n] \), written as \( D^{\ast} \), is called the \emph{dual} of \( D \).

	Let \( D = ([n], \mathcal{F}) \) be a set system and \( i \in [n] \).
	The \emph{loop complementation} of \( D \) on \( i \), denoted by \( D + i \), is defined as \( ([n], \mathcal{F}') \), where
	\[
	\mathcal{F}' = \mathcal{F} \triangle \{ I \cup \{i\} \mid I \in \mathcal{F} \text{ and } i \notin I \}.
	\]
	Brijder and Hoogeboom~\cite{Brijder2011} showed that for \( I \subseteq [n] \), \( D + I = ([n], \mathcal{F}') \), where \( X \in \mathcal{F}' \) if and only if
	\[
	\left| \{ Y \in \mathcal{F} \mid X \setminus I \subseteq Y \subseteq X \} \right|
	\]
	is odd. This implies that the loop complementation operation is independent of the sequence. Hence, for example, \( D + i + j = D + j + i = D + \{i, j\} \).
	
	Note that the class of delta-matroids is not closed under loop complementation, meaning that the loop complementation of a delta-matroid may not be a delta-matroid, as shown in the following example.
	
	\begin{example}
		Let \( D = ([n], \mathcal{F}) \) be a set system with \( \mathcal{F} = 2^{[n]} \setminus \{[n]\} \). It is easy to check that \( D \) is a delta-matroid. Now let \( n \geq 3 \) and consider \( \{1\} \subseteq [n] \). Then \( D + 1 = ([n], \mathcal{F}') \) is not a delta-matroid. Note that \( \mathcal{F}' \) contains no set \( I \) with \( |I| < n \) and \( 1 \in I \), but does contain \( \emptyset \) and \( [n] \). By taking \( X = \emptyset \), \( Y = [n] \), and \( x = 1 \in X \triangle Y \), there exists no \( j \in X \triangle Y \) such that \( X \triangle \{1, j\} \in \mathcal{F}' \).
	\end{example}
	
	For a delta-matroid \( D = ([n], \mathcal{F}) \), let \( \mathcal{F}_{\max}(D) \) and \( \mathcal{F}_{\min}(D) \) denote the sets of feasible sets of \( D \) with maximum and minimum cardinality, respectively. Define \( D_{\max} := ([n], \mathcal{F}_{\max}(D)) \) and \( D_{\min} := ([n], \mathcal{F}_{\min}(D)) \). Then \( D_{\max} \) is called the \emph{upper matroid} and \( D_{\min} \) the \emph{lower matroid}, with \( \mathcal{F}_{\max}(D) \) and \( \mathcal{F}_{\min}(D) \) as their respective families of bases.
	
	For a delta-matroid (or matroid) \( D = ([n], \mathcal{F}) \) and an element \( i \in [n] \), if \( i \) is contained in no feasible set (or basis, in the case of a matroid) of \( D \), then \( i \) is called a \emph{loop} of \( D \).
	
	\begin{definition}[\cite{Chun2019}]
		\normalfont
		Let \( D = ([n], \mathcal{F}) \) be a delta-matroid.
		\begin{enumerate}
			\item An element \( i \in [n] \) is a \emph{ribbon loop} if \( i \) is a loop in \( D_{\min} \).
			\item A ribbon loop \( i \) is \emph{non-orientable} if \( i \) is a ribbon loop in \( D \ast i \), and \emph{orientable} otherwise.
		\end{enumerate}
	\end{definition}

	\section{Group actions on set systems}\label{3}
	
	It has been shown in \cite{Brijder2011} that, for any fixed element \( i \in [n] \), the twist \( \ast i \) and loop complementation \( + i \) are involutions (i.e., of order 2). These two operations generate a group \( \mathcal{G} \) isomorphic to \( S_3 \), with the presentation:
	\[
	S_3 \cong \mathcal{G} := \langle \ast, + \mid \ast^2, +^2, (\ast+)^3 \rangle.
	\]
	We note that for any word \( g = g_1g_2\dots g_k \) in the alphabet \( \{\ast, +\} \), we define:
	\[
	D g i := ( \cdots (D g_k i) g_{k-1} i \cdots ) g_1 i.
	\]
	For instance, if \( g = +\ast \), then \( D (+\ast) i = (D \ast i) + i \). In particular, \( +\ast+ = \ast+\ast \in \mathcal{G} \), which is also an involution, denoted by \( \overline{\ast} \). In fact, it has been shown in \cite{Brijder2011} that for \( D = ([n], \mathcal{F}) \) and \( i \in [n] \), \( D \overline{\ast} i = ([n], \mathcal{F}') \), where
	\[
	\mathcal{F}' = \mathcal{F} \triangle \{ I \setminus \{i\} \mid I \in \mathcal{F} \text{ and } i \in I \}.
	\]
	Similarly, for \( I \subseteq [n] \), \( D \overline{\ast} I = ([n], \mathcal{F}') \), we have \( X \in \mathcal{F}' \) if and only if
	\[
	\left| \{ Y \in \mathcal{F} \mid X \subseteq Y \subseteq X \cup I \} \right|
	\]
	is odd. The six operations in \( \mathcal{G} \) are called \emph{invertible vertex flips}. A delta-matroid is said to be \emph{vf-safe} \cite{Chun2019F} if the application of every sequence of invertible vertex flips to each element results in a delta-matroid.
	
	Note that the invertible vertex flips commute on different elements. Therefore, for any distinct \( i, j \in [n] \) and any \( g_1, g_2 \in \mathcal{G} \), we have \( (D g_1 i) g_2 j = (D g_2 j) g_1 i \). For any subset \( I = \{i_1, i_2, \dots, i_k\} \subseteq [n] \) and any \( g \in \mathcal{G} \), we can unambiguously define
	\[
	D g I = ( \cdots (D g i_1) g i_2 \cdots ) g i_k.
	\]
	
	Let \( \mathcal{D}_{(n)} \) denote the set of all set systems with ground set \( [n] \). Recall that the group \( \mathcal{G} \) acts on any fixed element of a set system \( D \in \mathcal{D}_{(n)} \). This action extends naturally to a group action of \( \mathcal{G}^n \) on \( \mathcal{D}_{(n)} \): rather than acting on a single distinguished element, the group elements act independently and simultaneously on all elements. For any \( \hat{g} = (g_1, g_2, \dots, g_n) \in \mathcal{G}^n \) and any \( D \in \mathcal{D}_{(n)} \), the action of \( \hat{g} \) on \( D \) applies \( g_i \) to element \( i \) of \( D \). We view the indexing as a map, i.e., consider \( \hat{g} \) as \( \hat{g} : [n] \rightarrow \mathcal{G} \), so that the action applies \( \hat{g}(i) = g_i \) to element \( i \in [n] \). The multiplication in the group \( \mathcal{G}^n \) is denoted by \( \circ \) and defined as
	\[
	\hat{g} \circ \hat{h} = (g_1 h_1, g_2 h_2, \dots, g_n h_n),
	\]
	for any \( \hat{g} = (g_1, g_2, \dots, g_n), \hat{h} = (h_1, h_2, \dots, h_n) \in \mathcal{G}^n \).
	
	Since the invertible vertex flips commute on different elements, every sequence of invertible vertex flips admits a unique expression of the form
	\[
	D \prod_{i \in [n]} g_i(i),
	\]
	where \( g_i \in \mathcal{G} \). With the preceding notation, we can write the action of \( \hat{g} \in \mathcal{G}^n \) on \( D \) as
	\[
	\hat{g} \cdot D = D \Gamma(\hat{g}),
	\]
	where \( \Gamma(\hat{g}) = \prod_{i \in [n]} [\hat{g}(i)](i) = \prod_{i \in [n]} g_i(i) \).
	
	Note that for any \( \hat{g}, \hat{h} \in \mathcal{G}^n \), the group action satisfies
	\[
	(\hat{g} \circ \hat{h}) \cdot D = \hat{g} \cdot (\hat{h} \cdot D).
	\]
	This follows from the calculation below:
	\[
	\begin{aligned}
		(\hat{g} \circ \hat{h}) \cdot D
		&= D \prod_{i=1}^{n} (g_i h_i)(i) \\
		&= \left( D \prod_{i=1}^{n} h_i(i) \right) \prod_{i=1}^{n} g_i(i) \\
		&= \hat{g} \cdot (\hat{h} \cdot D).
	\end{aligned}
	\]
	This implies the operator identity
	\[
	(D \Gamma(\hat{h})) \Gamma(\hat{g}) = D \Gamma(\hat{g} \circ \hat{h}).
	\]
	
	A set system \( D \) is said to be \emph{canonical self-duality} if there exists a sequence of invertible vertex flips applied to certain elements such that the resulting set system equals \( D \). For example, consider \( D = ([2], \{\{1\}, \{2\}\}) \), and let \( X = \{1, 2\} \subseteq [2] \). After performing a twist operation on \( D \) with respect to \( X \), the resulting delta-matroid equals \( D \), that is, \( D \ast X = D \).
	
	However, in practical research, we often consider a more relaxed form of self-duality, where the result of a sequence of invertible vertex flip operations is not directly equal to the original set system but becomes equal after some relabeling of elements (i.e., permutations). For instance, consider \( D = ([2], \{\emptyset, \{1\}, \{1, 2\}\}) \) and let \( X = \{1, 2\} \subseteq [2] \). Then \( D \ast X = ([2], \{\emptyset, \{2\}, \{1, 2\}\}) \). It can be observed that \( D \ast X \) equals \( D \) under the permutation that maps \( 1 \) to \( 2 \) and \( 2 \) to \( 1 \). Based on this, we introduce the concept of natural self-duality as follows: a set system \( D \) is said to be \emph{natural self-duality} if there exists a sequence of invertible vertex flips applied to elements such that the resulting set system equals \( D \) under some permutation.
	
	Due to the distinction between canonical and natural self-duality, we begin by considering permutations on set systems. Let \( D = ([n], \mathcal{F}) \) be a set system and \( \pi \) be a permutation in \( S_n \). Then we define the action of \( \pi \) on \( D \) as:
	\begin{align*}
		\pi: D &\mapsto D_\pi, \\
		i &\mapsto \pi(i),
	\end{align*}
	where
	\[
	D_\pi = ([n], \{\pi(I) \mid I \in \mathcal{F}\}),
	\]
	and for \( I = \{i_1, i_2, \dots, i_k\} \), we have \( \pi(I) = \{\pi(i_1), \pi(i_2), \dots, \pi(i_k)\} \).
	
	Note that for a set system \( D \in \mathcal{D}_{(n)} \) and permutations \( \pi_1, \pi_2 \in S_n \), the composition satisfies:
	\[
	(\pi_1 \pi_2)D = \pi_1(\pi_2 D) = (D_{\pi_2})_{\pi_1}.
	\]
	
	For any \( \hat{g} = (g_1, g_2, \dots, g_n) \in \mathcal{G}^n \) and \( \pi \in S_n \), we define \[ \hat{g}\pi^{-1}(i) := \hat{g}(\pi^{-1}(i)) = g_{\pi^{-1}(i)}. \] It follows that
	\[
	\hat{g}\pi^{-1} = (g_{\pi^{-1}(1)}, g_{\pi^{-1}(2)}, \dots, g_{\pi^{-1}(n)}).
	\]
	In fact, \( \hat{g}\pi^{-1} \) is the reindexing of \( \hat{g} \) by \( \pi^{-1} \), which corresponds to permuting the indices of \( [n] \).
	
	Specifically, for the operator \( \Gamma \), by substituting \( j = \pi^{-1}(i) \) and replacing \( j \) by \( i \) in the product, we obtain:
	\[
	\Gamma(\hat{g}\pi^{-1}) = \prod_{i \in [n]} g_{\pi^{-1}(i)}(i) = \prod_{i \in [n]} g_i(\pi(i)).
	\]
	
	We now introduce a new group action on \( \mathcal{D}_{(n)} \) by utilizing the semidirect product of the symmetric group \( S_n \) and the group \( \mathcal{G}^n \).
	
	\begin{proposition}
		Let \(\phi: S_n \rightarrow \operatorname{Aut}(\mathcal{G}^n)\) be defined by \(\phi(\pi) \mapsto \phi_\pi\), where \(\phi_\pi(\hat{g}) = \hat{g}\pi^{-1}\). Then \(\phi\) is a homomorphism, and the semidirect product \(\mathcal{G}^n \rtimes_\phi S_n\) acts on \(\mathcal{D}_{(n)}\) by \((\hat{g},\pi)D = \hat{g} \cdot (\pi D) = \hat{g} \cdot D_\pi = D_\pi \Gamma(\hat{g})\).
	\end{proposition}
	
	\begin{proof}
		For any \(i \in [n]\) and \( \pi_1, \pi_2 \in S_n \),
		\[
		\hat{g}(\pi_2^{-1}\pi_1^{-1})(i) = \hat{g}\left((\pi_2^{-1}\pi_1^{-1})(i)\right) = \hat{g}\left(\pi_2^{-1}(\pi_1^{-1}(i))\right).
		\]
		Let \(j = \pi_1^{-1}(i)\), then
		\[
		\hat{g}\left(\pi_2^{-1}(\pi_1^{-1}(i))\right) = \hat{g}\left(\pi_2^{-1}(j)\right) = \hat{g}\pi_2^{-1}(j) = \hat{g}\pi_2^{-1}(\pi_1^{-1}(i)) = \left((\hat{g}\pi_2^{-1})\pi_1^{-1}\right)(i).
		\]
		It follows that \(\phi\) is a homomorphism since
		\[
		\phi_{\pi_1\pi_2}(\hat{g}) = \hat{g}(\pi_1\pi_2)^{-1} = \hat{g}(\pi_2^{-1}\pi_1^{-1}) = (\hat{g}\pi_2^{-1})\pi_1^{-1} = \phi_{\pi_1}\phi_{\pi_2}(\hat{g}).
		\]
		
		Thus, we have a well-defined semidirect product \(\mathcal{G}^n \rtimes_\phi S_n\) with multiplication given by
		\[
		(\hat{g},\pi_1)(\hat{h},\pi_2) = (\hat{g} \circ \phi_{\pi_1}(\hat{h}), \pi_1\pi_2) = (\hat{g} \circ \hat{h}\pi_1^{-1}, \pi_1\pi_2).
		\]
		
		It remains to prove that the semidirect product group \(\mathcal{G}^n \rtimes_\phi S_n\) acts on \(\mathcal{D}_{(n)}\), i.e., to verify:
		\begin{itemize}
			\item[(1)] For any \(D \in \mathcal{D}_{(n)}\) and \((\hat{g},\pi) \in \mathcal{G}^n \rtimes_\phi S_n\), we have \((\hat{g},\pi)D \in \mathcal{D}_{(n)}\) and \((\mathbf{1},\iota)D = D\), where \(\mathbf{1}\) and \(\iota\) are the identities of \(\mathcal{G}^n\) and \(S_n\), respectively, and \((\mathbf{1},\iota)\) is the identity element of the semidirect product;
			\item[(2)] For any \(D \in \mathcal{D}_{(n)}\) and \((\hat{g},\pi_1), (\hat{h},\pi_2) \in \mathcal{G}^n \rtimes_\phi S_n\),
			\[
			(\hat{g},\pi_1)\left((\hat{h},\pi_2)D\right) = \left((\hat{g},\pi_1)(\hat{h},\pi_2)\right)D.
			\]
		\end{itemize}
		
		Item (1) is evident, so we only verify (2). By the definition of the semidirect product action:
		\[
		\begin{aligned}
			(\hat{g},\pi_1)\left((\hat{h},\pi_2)D\right)
			&= (\hat{g},\pi_1)\left(D_{\pi_2}\Gamma(\hat{h})\right) \\
			&= (\hat{g},\pi_1)\left(D_{\pi_2} \prod_{i\in [n]} h_i(i)\right) \\
			&= \hat{g} \cdot \left(\pi_1\left[D_{\pi_2} \prod_{i\in [n]} h_i(i)\right]\right) \\
			&= \hat{g} \cdot \left((D_{\pi_2})_{\pi_1} \prod_{i\in [n]} h_i(\pi_1(i))\right) \\
			&= D_{\pi_1\pi_2} \Gamma(\hat{h}\pi_1^{-1}) \Gamma(\hat{g}) \\
			&= D_{\pi_1\pi_2} \Gamma(\hat{g} \circ \hat{h}\pi_1^{-1}) \\
			&= \left((\hat{g},\pi_1)(\hat{h},\pi_2)\right)D.
		\end{aligned}
		\]
	\end{proof}
	
	Note that \( (\hat{g},\pi)^{-1} = (\hat{g}^{-1}\pi, \pi^{-1}) \), since
	\[
	(\hat{g},\pi)(\hat{g}^{-1}\pi, \pi^{-1}) = (\mathbf{1},\iota).
	\]
	
	We conclude this section by formalizing the notation for orbits and stabilizers under group actions.
	The orbit of \( D \) is given by
	\[
	\operatorname{Orb}(D) = \left\{ (\hat{g},\pi)D \bigm| (\hat{g},\pi) \in \mathcal{G}^n \rtimes_\phi S_n \right\}.
	\]
	In particular, we often focus on the action of the subgroup \( \mathcal{G}^n \rtimes_\phi \{\iota\} \), where \(\iota\) is the identity element of \(S_n\), and denote its orbit by
	\[
	\operatorname{Orb}_\iota(D) = \left\{ (\hat{g},\iota)D \bigm| (\hat{g},\iota) \in \mathcal{G}^n \rtimes_\phi \{\iota\} \right\}.
	\]
	
	\begin{definition}
		\normalfont
		A set system \( D \) is said to be \emph{self-$\hat{g}$ via $\pi$} if \( \hat{g} \) is not the identity and there exists a \( \pi \in S_n \) such that \( (\hat{g},\pi)D = D \). Furthermore, \( D \) is said to be \emph{canonically self-$\hat{g}$} if \( \pi = \iota \).
	\end{definition}
	
	To study the self-$\hat{g}$ property of a set system, i.e., to find some \( (\hat{g},\pi) \in \mathcal{G}^n \rtimes_\phi S_n \) such that \( (\hat{g},\pi)D = D \), we are led to investigate the orbits and stabilizers of this group action. In Section 5, we shall provide a characterization of the orbits of a given vf-safe delta-matroid \( D \) under the action of the group \( \mathcal{G}^n \rtimes_\phi S_n \).
	
	We define an element \( \hat{g} \in \mathcal{G}^n \) to be \emph{uniform} if all entries of \( \hat{g} \) are the same, and this common entry is a non-identity element of \( \mathcal{G} \). Our primary focus will be on the case where a set system \( D \) is self-$\hat{g}$ with \( \hat{g} \) being uniform.
	
	\begin{definition}
		\normalfont
		A set system \( D \) is said to be \emph{self-twual via $\pi$} if there exists a uniform \( \hat{g} \) and some \( \pi \in S_n \) such that \( (\hat{g},\pi)D = D \). Furthermore, \( D \) is said to be \emph{canonically self-twual} if \( \pi = \iota \).
	\end{definition}
	
	For any uniform \( \hat{g} \), we can write \( \hat{g} = \mathbf{1}g \), where \( g \in \mathcal{G} \), \( g \) is not the identity element of \( \mathcal{G} \), and \( \mathbf{1} \) is the identity of \( \mathcal{G}^n \). A set system \( D \) is defined to be \emph{canonically self-twist} if \( (\mathbf{1}\ast, \iota)D = D \), and \emph{canonically self-loop complementation} if \( (\mathbf{1}+, \iota)D = D \).
	
	\section{Propagation of self-twual in orbits}\label{4}
	
	In this section, we show that if a set system \( D \) is self-$\hat{g}$, then any
	\( D' \in \operatorname{Orb}(D) \) is also self-$\hat{g}'$ for some \( \hat{g}' \).
	Thus, once any set system is self-$\hat{g}$, this property propagates through its entire orbit.
	
	\begin{theorem}\label{near conjugate}
		If \( D \) is self-\(\hat{g}\) via \( \mu \) and \( D' \in \operatorname{Orb}(D) \) with \( (\hat{h},\pi)D = D' \), then \( D' \) is self-\(\hat{g}'\) via \( \mu' \), where
		\[
		\hat{g}' = \hat{h} \circ \hat{g}\pi^{-1} \circ \hat{h}^{-1}\mu'^{-1},
		\]
		and \( \mu' = \pi\mu\pi^{-1} \).
	\end{theorem}
	
	\begin{proof}
		It suffices to verify that \( (\hat{g}',\mu')D' = D' \). Since \( (\hat{g},\mu)D = D \) and \( (\hat{h},\pi)D = D' \), it is equivalent to show that
		\[
		(\hat{g}',\mu')[(\hat{h},\pi)D] = (\hat{h},\pi)[(\hat{g},\mu)D].
		\]
		
		By the compatibility of group actions, this reduces to verifying the equality
		\begin{align}\label{change}
			(\hat{g}',\mu')(\hat{h},\pi) = (\hat{h},\pi)(\hat{g},\mu)
		\end{align}
		in the semidirect product group \( \mathcal{G}^n \rtimes_\phi S_n \). We compute:
		\[
		\begin{aligned}
			(\hat{g}',\mu')(\hat{h},\pi)
			&= (\hat{h} \circ \hat{g}\pi^{-1} \circ \hat{h}^{-1}\mu'^{-1}, \mu')(\hat{h},\pi) \\
			&= (\hat{h} \circ \hat{g}\pi^{-1} \circ \hat{h}^{-1}\mu'^{-1} \circ \hat{h}\mu'^{-1}, \mu'\pi) \\
			&= (\hat{h} \circ \hat{g}\pi^{-1} \circ (\hat{h}^{-1} \circ \hat{h})\mu'^{-1}, \pi\mu) \\
			&= (\hat{h} \circ \hat{g}\pi^{-1}, \pi\mu) \\
			&= (\hat{h},\pi)(\hat{g},\mu).
		\end{aligned}
		\]
		
		This process can be represented by the following commutative diagram:
		\[
		\begin{tikzcd}[row sep=1.5cm, column sep=4.5cm]
			D' \arrow[r, "{(\hat{g}',\mu')}"] & D' \\
			D \arrow[r, "{(\hat{g},\mu)}"] \arrow[u, "{(\hat{h},\pi)}"] & D \arrow[u, "{(\hat{h},\pi)}"]
		\end{tikzcd}
		\]
	\end{proof}
	
	When specializing the theorem to the case where \( D \) is canonically self-\(\hat{g}\),
	every set system in the orbit \( \operatorname{Orb}_\iota(D) \) is also canonically
	self-\(\hat{g}'\) for some \( \hat{g}' \) conjugate to \( \hat{g} \).
	
	\begin{corollary}\label{conjugate}
		If \( D \) is canonically self-\(\hat{g}\) and \( D' = (\hat{h},\iota)D \in \operatorname{Orb}_\iota(D) \),
		then \( D' \) is also canonically self-\(\hat{g}'\), where \( \hat{g}' \) is the conjugate
		of \( \hat{g} \) by \( \hat{h} \), explicitly given by \( \hat{g}' = \hat{h} \circ \hat{g} \circ \hat{h}^{-1} \).
	\end{corollary}
	
	Theorem~\ref{near conjugate} implies that if we find that any set system \( D \) is
	self-\(\hat{g}\) via \( \mu \), we can then test for the suitability of \( \hat{h} \) and \( \pi \)
	such that
	\[
	\hat{g}' = \hat{h} \circ \hat{g}\pi^{-1} \circ \hat{h}^{-1}\pi\mu^{-1}\pi^{-1}
	\]
	is uniform, and thus identify a self-twual set system in \( \operatorname{Orb}(D) \).
	
	\begin{example}
		Let \( D = ([3], \{\{3\}, \{1,3\}, \{2,3\}\}) \). It is easy to verify that \( D \) is self-\(\hat{g}\) with \(\hat{g} = (\ast, +, +)\) via \(\iota\).
		Let \(\hat{h} = (+, \ast, \ast)\), then
		\[
		D' = (\hat{h},\iota)D = D + 1 \ast 2 \ast 3 = ([3], \{\emptyset, \{1\}, \{2\}\}),
		\]
		and
		\[
		\hat{g}' = \hat{h} \circ \hat{g} \circ \hat{h}^{-1} = (\overline{\ast}, \overline{\ast}, \overline{\ast}).
		\]
		Then \( (\hat{g}',\iota)D' = D' \overline{\ast} \{1,2,3\} = D' \), so \( D' \) is canonically self-twual.
	\end{example}
	
	Since \( \mathcal{G} \) is non-commutative, we write \( \prod_{i=m}^{1} g_i \) in the following theorem to indicate the product \( g_m g_{m-1} \dots g_1 \). We also denote by \( |g| \) the order of a group element \( g \in \mathcal{G} \). The following theorem may assist in constructing self-twual set systems.
	
	\begin{theorem}\label{Orb}
		Let \( S \) be a set system, \( g \in \mathcal{G} \), and \( \mu \) a permutation. Then the following are equivalent:
		\begin{itemize}
			\item[$(1)$] There exists a set system \( D \in \operatorname{Orb}(S) \) that is self-\(\hat{g}\) via the permutation \( \mu \), where \( \hat{g} = (g,g,\dots,g) \) is uniform;
			\item[$(2)$] For every set system \( D' \in \operatorname{Orb}(S) \), \( D' \) is self-\(\hat{g}' = (g_1,g_2,\dots,g_n)\) via \( \mu' \), where \( \mu' = \pi\mu\pi^{-1} \) for some \( \pi \), and \( \hat{g}' \) satisfies the following condition: if \( C = (c_1,c_2,\dots,c_m) \) is a cycle in the cycle decomposition of \( \mu' \), then
			\[
			\left|\prod_{i=m}^{1} g_{c_i}\right| = |g^m|;
			\]
			\item[$(3)$] There exists a set system \( D' \in \operatorname{Orb}(S) \) that is self-\(\hat{g}' = (g_1,g_2,\dots,g_n)\) via \( \mu' \), where \( \mu' = \mu \), and \( \hat{g}' \) satisfies the following condition: if \( C = (c_1,c_2,\dots,c_m) \) is a cycle in the cycle decomposition of \( \mu' \), then
			\[
			\left|\prod_{i=m}^{1} g_{c_i}\right| = |g^m|.
			\]
		\end{itemize}
	\end{theorem}
	
	\begin{proof}
		(1) \(\Rightarrow\) (2): Let \( D, D' \in \operatorname{Orb}(S) \), and suppose \( D \) is self-\(\hat{g} = (g,g,\dots,g)\) via \( \mu \). Then there exists \( (\hat{h},\pi) \) such that \( (\hat{h},\pi)D = D' \). By Theorem~\ref{near conjugate}, \( D' \) is self-\(\hat{g}'\) via \( \mu' = \pi\mu\pi^{-1} \), where
		\[
		\hat{g}' = \hat{h} \circ \hat{g}\pi^{-1} \circ \hat{h}^{-1}\mu'^{-1},
		\]
		and by equation (\ref{change}) in Theorem~\ref{near conjugate}, we have
		\[
		(\hat{h},\pi)(\hat{g},\mu) = (\hat{g}',\mu')(\hat{h},\pi).
		\]
		This implies that
		\[
		\begin{aligned}
			(\hat{g},\mu) &= (\hat{h},\pi)^{-1}(\hat{g}',\mu')(\hat{h},\pi) \\
			&= (\hat{h}^{-1}\pi, \pi^{-1})(\hat{g}' \circ \hat{h}\mu'^{-1}, \mu'\pi) \\
			&= \left((\hat{h}^{-1} \circ \hat{g}' \circ \hat{h}\mu'^{-1})\pi, \pi^{-1}\mu'\pi\right).
		\end{aligned}
		\]
		Thus, we have \( \hat{g}\pi^{-1} = \hat{h}^{-1} \circ \hat{g}' \circ \hat{h}\mu'^{-1} \). Since \( \hat{g} \) is uniform, it follows that \( \hat{g}\pi^{-1} = \hat{g} \), leading to
		\[
		\hat{g} = \hat{h}^{-1} \circ \hat{g}' \circ \hat{h}\mu'^{-1}.
		\]
		Writing \( \hat{g}' = (g_1, g_2, \dots, g_n) \) and \( \hat{h} = (h_1, h_2, \dots, h_n) \), we obtain for each \( i \):
		\[
		g = h_i^{-1} g_i h_{\mu'^{-1}(i)}.
		\]
		Let \( C = (c_1, c_2, \dots, c_m) \) be a cycle in the cycle decomposition of \( \mu' \). For all \( i \in \{1, 2, \dots, m\} \), we have
		\[
		g = h_{c_i}^{-1} g_{c_i} h_{c_{i-1}},
		\]
		where \( c_0 = c_m \). Rearranging gives
		\[
		g_{c_i} = h_{c_i}g h_{c_{i-1}}^{-1}.
		\]
		Thus,
		\[
		\prod_{i=m}^{1} g_{c_i} =(h_{c_m}g h_{c_{m-1}}^{-1})(h_{c_{m-1}}g h_{c_{m-2}}^{-1})\cdots  (h_{c_1}g h_{c_{m}}^{-1})=h_{c_m} g^m h_{c_m}^{-1},
		\]
		and the conclusion holds because conjugation preserves order in groups.
		
		(2) \(\Rightarrow\) (3): Suppose \( D' \) is self-\(\hat{g}'\) via \( \mu' = \pi\mu\pi^{-1} \).
		Define \( D'' = (\mathbf{1}, \pi^{-1})D' \). By Theorem~\ref{near conjugate}, \( D'' \) is
		self-\(\hat{g}'\pi\) via \( \mu \). For each cycle \( C = (c_1, c_2, \dots, c_m) \) in the
		cycle decomposition of \( \mu \), since \( \mu' = \pi\mu\pi^{-1} \), there exists a
		corresponding cycle \( C' = (\pi(c_1), \pi(c_2), \dots, \pi(c_m)) \) in the cycle
		decomposition of \( \mu' \). Therefore,
		\[
		\left| \prod_{i=m}^{1} (\hat{g}'\pi)(c_i) \right| =
		\left| \prod_{i=m}^{1} \hat{g}'(\pi(c_i)) \right| =
		\left| \prod_{i=m}^{1} g_{\pi(c_i)} \right| = |g^m|,
		\]
		and conclusion (3) holds.
		
		(3) \(\Rightarrow\) (1): We aim to construct a \( \hat{h} \in \mathcal{G}^n \) such that
		\[
		(\hat{h},\iota)(\hat{g}',\mu)(\hat{h},\iota)^{-1} = (\hat{g},\mu).
		\]
		Then by setting \( D = (\hat{h},\iota)D' \), we obtain
		\[
		(\hat{g},\mu)D = (\hat{g},\mu)(\hat{h},\iota)D' = (\hat{h},\iota)(\hat{g}',\mu)D' = (\hat{h},\iota)D' = D,
		\]
		hence (1) holds.
		
		For each cycle \( C = (c_1, c_2, \dots, c_m) \) in the cycle decomposition of \( \mu \), since
		\[
		\left| \prod_{i=m}^{1} g_{c_i} \right| = |g^m|
		\]
		and \( \prod_{i=m}^{1} g_{c_i}, g^m \in \mathcal{G} \cong S_3 \), then \( \prod_{i=m}^{1} g_{c_i} \) and \( g^m \) are conjugate (as elements of the same order in \( S_3 \) are conjugate). Thus, there exists an \( h_{c_m} \in \mathcal{G} \) such that
		\begin{align}\label{gm}
			h_{c_m} \left( \prod_{i=m}^{1} g_{c_i} \right) h_{c_m}^{-1} = g^m.
		\end{align}
		
		For \( 1 \leq i < m \), recursively define
		\begin{align}\label{eq}
			h_{c_i} = g^{-1} h_{c_{i+1}} g_{c_{i+1}}.
		\end{align}
		By induction, this recurrence relation yields
		\[
		h_{c_i} = g^{-(m-i)} h_{c_m} \left( \prod_{j=m}^{i+1} g_{c_j} \right).
		\]
		In particular, setting \( i = 1 \) and using the identity
		\[
		h_{c_m} = g^m h_{c_m} \left( \prod_{i=1}^{m} g_{c_i}^{-1} \right)
		\]
		(which follows from Equation (\ref{gm})), we derive
		\[
		h_{c_1} = g^{-(m-1)} h_{c_m} \prod_{j=m}^{2} g_{c_j} = g h_{c_m} g_{c_1}^{-1}.
		\]
		Rearranging this relation yields
		\begin{align}\label{m}
			g = h_{c_1} g_{c_1} h_{c_m}^{-1} = h_{c_1} g_{c_1} h_{\mu^{-1}(c_1)}^{-1}.
		\end{align}
		
		Then, by combining Equations (\ref{eq}) and (\ref{m}), we find that for all \( i \), the following holds:
		\begin{align}\label{element}
			g = h_{c_{i+1}} g_{c_{i+1}} h_{c_i}^{-1} = h_{c_{i+1}} g_{c_{i+1}} h_{\mu^{-1}(c_{i+1})}^{-1}.
		\end{align}
		
		Due to the arbitrariness of the cycle, we can choose suitable \( h_{c_m} \) for each cycle \( C = (c_1, c_2, \dots, c_m) \) in \( \mu \). Thus Equation (\ref{element}) holds for all entries, which implies
		\[
		\hat{g} = \hat{h} \circ \hat{g}' \circ \hat{h}^{-1}\mu^{-1}.
		\]
		Therefore, we have
		\[
		\begin{aligned}
			(\hat{h},\iota)(\hat{g}',\mu)(\hat{h},\iota)^{-1}
			&= (\hat{h} \circ \hat{g}', \mu)(\hat{h}^{-1}, \iota) \\
			&= \left( \hat{h} \circ \hat{g}' \circ \hat{h}^{-1}\mu^{-1}, \mu \right) \\
			&= (\hat{g}, \mu).
		\end{aligned}
		\]
	\end{proof}
	
	Specializing the theorem, we obtain the following corollary.
	
	\begin{corollary}\label{Orb_1}
		Let \( S \) be a set system. Then the following are equivalent:
		\begin{itemize}
			\item[$(1')$] There exists a set system \( D \in \operatorname{Orb}_\iota(S) \) that is canonically self-\(\hat{g} = (g,g,\dots,g)\), where \( |g| = 2 \) (respectively, \( |g| = 3 \));
			\item[$(2')$] For every set system \( D' \in \operatorname{Orb}_\iota(S) \), \( D' \) is canonically self-\(\hat{g}' = (g_1,g_2,\dots,g_n)\), where \( |g_i| = 2 \) (respectively, \( |g_i| = 3 \)) for all \( i \);
			\item[$(3')$] There exists a set system \( D' \in \operatorname{Orb}_\iota(S) \) that is canonically self-\(\hat{g}' = (g_1,g_2,\dots,g_n)\), where \( |g_i| = 2 \) (respectively, \( |g_i| = 3 \)) for all \( i \).
		\end{itemize}
		
		\begin{proof}
			Setting \( \mu = \iota \), the cycle decomposition of \( \iota \) is given by the product of \( n \) disjoint 1-cycles: \( (1)(2)\cdots(n) \). Since \( \operatorname{Orb}_\iota(S) \subseteq \operatorname{Orb}(S) \), the following implications hold by Theorem~\ref{Orb}:
			\[
			(1') \Rightarrow (1) \Rightarrow (2) \Rightarrow (2') \Rightarrow (3').
			\]
			Thus, it remains to prove \( (3') \Rightarrow (1') \). In the proof of Theorem~\ref{Orb} (3) \(\Rightarrow\) (1), we set \( D = (\hat{h},\iota)D' \). Since \( D' \in \operatorname{Orb}_\iota(S) \), it follows that \( D \in \operatorname{Orb}_\iota(S) \). Therefore, the conclusion holds.
		\end{proof}
	\end{corollary}
	
	For a set system \( D \) that is self-\(\hat{g}\) via \( \mu \), Theorem~\ref{Orb} and Corollary~\ref{Orb_1} state that in some cases, analyzing the relationship between the orders of the \( g_i \) and the cycle decomposition of \( \mu \) suffices to determine whether there exists a self-twual set system in the orbit of \( D \).
	
	\begin{example}
		Let \( D = ([3], \{\{3\}, \{1,3\}, \{2,3\}\}) \). It is easy to verify that \( D \) is self-\(\hat{g}\) with \( \hat{g} = (\ast, +, +) \) via \( \iota \). The cycle decomposition of \( \iota \) is \( (1)(2)(3) \), and both \( \ast \) and \( + \) have order 2. Thus, there exists a set system \( D' \in \operatorname{Orb}_\iota(D) \) such that \(  D' \) is self-twual. In fact, let
		\[
		D' = D + 1 \ast 2 \ast 3 = ([3], \{\emptyset, \{1\}, \{2\}\})
		\]
		and \( \hat{g}' = (\overline{\ast}, \overline{\ast}, \overline{\ast}) \). Then
		\[
		(\hat{g}',\iota)D' = D' \overline{\ast} \{1,2,3\} = D',
		\]
		so \( D' \) is self-\((\overline{\ast}, \overline{\ast}, \overline{\ast})\) via \( \iota \).
	\end{example}
	
	\begin{remark}
		
		Abrams and Ellis-Monaghan \cite{Abrams2022} showed that the orbit of a connected ribbon graph must contains an orientable bouquet, which is a ribbon graph with a single vertex, and prove that the (partial) twuality properties of the bouquet propagate throughout its orbit.
		In fact, this result can be generalized to vf-safe delta-matroids. We only need to demonstrate that the orbit of a vf-safe delta-matroid contains a vf-safe delta-matroid with only orientable ribbon loops. To elaborate, for a vf-safe delta-matroid \( D = ([n], \mathcal{F}) \), let \( F \in \mathcal{F}_{\min}(D) \) (where \( F \) may be the empty set). Then \( D \ast F \) is normal, so each element in \( D \ast F \) is a ribbon loop. Let \( \{i_1\}, \{i_2\}, \dots, \{i_k\} \) denote all feasible sets of size \( 1 \) in \( D \ast F \); it follows that \( D' = D \ast F + \{i_1\} + \{i_2\} + \dots + \{i_k\} \) is normal with no feasible sets of size \( 1 \). Thus, \( D' \in \operatorname{Orb}(D) \), and for any \( i \in [n] \), \( i \) is an orientable ribbon loop of \( D' \).
	\end{remark}

	\section{Orbits characterization via multimatroids}\label{5}
	
	In this section, for a vf-safe delta-matroid \( D \), we study the orbit of \( D \) using multimatroids, which were introduced by Bouchet \cite{Bouchet1997}.
	
	A \emph{carrier} is a pair \( (U, \Omega) \), where \( \Omega \) is a partition of a finite set \( U \). An \((n,k)\)-carrier is a carrier \( (U, \Omega) \) such that \( |\Omega| = n \) and for all \( \omega \in \Omega \), \( |\omega| = k \). Each \( \omega \in \Omega \) is called a \emph{skew class}, and a subset \( p \subseteq \omega \) with \( |p| = 2 \) is called a \emph{skew pair} of \( \omega \). A \emph{transversal} (respectively, \emph{subtransversal}) \( T \) of \( \Omega \) is a subset of \( U \) such that \( |T \cap \omega| = 1 \) (respectively, \( |T \cap \omega| \leq 1 \)) for all \( \omega \in \Omega \). The set of all transversals (respectively, subtransversals) of \( \Omega \) is denoted by \( \mathcal{T}(\Omega) \) (respectively, \( \mathcal{S}(\Omega) \)).
	
	\begin{definition}[\cite{Bouchet1997}]
		\normalfont
		A \emph{multimatroid} \( Z \) (described by its independent sets) is a triple \( (U, \Omega, \mathcal{I}) \), where \( (U, \Omega) \) is a carrier and \( \mathcal{I} \subseteq \mathcal{S}(\Omega) \) satisfies the following axioms:
		\begin{itemize}
			\item[(1)] For each transversal \( T \in \mathcal{T}(\Omega) \), \( (T, \mathcal{I} \cap 2^T) \) is a matroid (described by its independent sets);
			\item[(2)] For any \( I \in \mathcal{I} \) and any skew pair \( p = \{x, y\} \) of some skew class \( \omega \in \Omega \) with \( \omega \cap I = \emptyset \), \(I\cup x\in \mathcal{I}\) or \(I\cup y\in \mathcal{I}\).
		\end{itemize}
	\end{definition}
	
	Each \( I \in \mathcal{I} \) of \( Z \) is referred to as an \emph{independent set} of \( Z \).
	The set of maximal independent sets \( \max(\mathcal{I}) \) of \( Z \) (with respect to inclusion)
	form the set of \emph{bases}, denoted by \( \mathcal{B}_Z \). We define
	\[ \mathcal{B}_Z^* := \{ I \mid I \subseteq B \in \mathcal{B}_Z \}. \]
	For any \( B \in \mathcal{B}_Z \), since \( B \in \mathcal{I} \cap 2^B \),
	\( (B, \mathcal{I} \cap 2^B) \) is a matroid (described by its independent sets)
	if and only if for any \( I \subseteq B \), we have \( I \in \mathcal{I} \).
	Thus, \( \mathcal{I} = \mathcal{B}_Z^* \).
	Therefore, for a fixed carrier, the set of bases uniquely determine \( Z \).
	
	For any subset \( X \subseteq U \), the \emph{restriction} of \( Z \) to \( X \),
	denoted by \( Z[X] \), is the multimatroid \( (X, \Omega', \mathcal{I} \cap 2^X) \),
	where \[ \Omega' = \{\omega \cap X \mid \omega \cap X \neq \emptyset, \omega \in \Omega\}. \]

	A \emph{$k$-matroid} \( Z \) is a multimatroid with carrier \( (U, \Omega) \) such that
	\( |\omega| = k \) for all \( \omega \in \Omega \). For \( k>1 \) and any \( k \)-matroid \( Z \),
	the base family \( \mathcal{B}_Z \) is necessarily contained in \( \mathcal{T}(\Omega) \). For \( k > 1 \), a \( k \)-matroid \( Z \) is \emph{tight} if for every basis \( X \in \mathcal{B}_Z \)
	and every skew class \( \omega \in \Omega \), exactly one of the transversals
	\( (X \setminus \omega) \cup \{u\} \) for \( u \in \omega \) is not a basis of \( Z \).
	
	We recall from \cite{Bouchet1987}, \cite{Bouchet2001}, and \cite{Brijder2014} that
	\( 2 \)-matroids correspond to delta-matroids, tight \( 2 \)-matroids correspond to
	even delta-matroids, and tight \( 3 \)-matroids correspond to vf-safe delta-matroids.
	
	Let \( (U, \Omega) \) be an \( (n,k) \)-carrier. A \emph{transversal \(k\)-tuple} of \( (U, \Omega) \)  is a sequence \( (T_1, T_2, \dots, T_k) \) of mutually disjoint transversals of \( \Omega \)
	such that \( \bigcup_{i=1}^{k} T_i = U \). A \emph{projection} of \( (U, \Omega) \) is a surjective function \( \sigma: U \rightarrow [n] \)
	such that for any \( x, y \in U \), \( \sigma(x) = \sigma(y) \) if and only if \( x \) and \( y \)
	belong to the same skew class \( \omega \in \Omega \).
	
	Let \( Z \) be a \( 2 \)-matroid with an \( (n,2) \)-carrier \( (U, \Omega) \).
	For a transversal \( 2 \)-tuple \( \tau = (T_1, T_2) \) of \( (U, \Omega) \) and a projection
	\( \sigma: U \rightarrow [n] \), the set system \( ([n], \mathcal{F}) \) with
	\[
	\mathcal{F} = \{\sigma(X \cap T_2) \mid X \in \mathcal{B}_Z\}
	\]
	is a delta-matroid, denoted by \( D_{Z,\tau,\sigma} \); see \cite{Brijder2014}.
	
	We start from a tight \( 3 \)-matroid and obtain a unique vf-safe delta-matroid via a given transversal
	\( 3 \)-tuple \( \tau \) and a projection \( \sigma \).
	
	Let \( Z = (U, \Omega, \mathcal{I}) \) be a tight \( 3 \)-matroid with an \( (n,3) \)-carrier \( (U, \Omega) \).
	Fix a transversal \( 3 \)-tuple \( \tau = (T_1, T_2, T_3) \) of \( (U, \Omega) \) and a projection
	\( \sigma: U \rightarrow [n] \). Since the restriction \( Z[T_1 \cup T_2] \) is a \( 2 \)-matroid,
	the delta-matroid \( D_{Z[T_1 \cup T_2], (T_1, T_2), \sigma} \) is well-defined. We denote this
	delta-matroid simply by \( D_{Z,\tau,\sigma} \). Brijder and Hoogeboom showed in \cite{Brijder2014}
	that \( D_{Z,\tau,\sigma} \) is vf-safe.
	
	Note that the set of bases of \( Z[T_1 \cup T_2] \), denoted by \( \mathcal{B}_{Z[T_1 \cup T_2]} \), is
	\[
	\mathcal{B}_{Z[T_1 \cup T_2]} = \{ B \mid B \in \max(\mathcal{I} \cap 2^{T_1 \cup T_2}) \}.
	\]
	Since \( Z[T_1 \cup T_2] \) is a \( 2 \)-matroid, for any \( I \in \mathcal{I} \cap 2^{T_1 \cup T_2} \)
	and any skew pair \( p = \{x, y\} \subseteq T_1 \cup T_2 \) of some \( \omega \in \Omega \) with
	\( \omega \cap I = \emptyset \), we have \( I \cup x \in \mathcal{I} \) or \( I \cup y \in \mathcal{I} \).
	Consequently, the set of maximal independent sets of \( Z[T_1 \cup T_2] \) satisfies:
	\[
	\max(\mathcal{I} \cap 2^{T_1 \cup T_2}) = \{ B \in \mathcal{B}_Z \mid B \subseteq T_1 \cup T_2 \}.
	\]
	Thus, the base family of the restricted multimatroid \( Z[T_1 \cup T_2] \) is given by:
	\begin{align}\label{restrict}
		\mathcal{B}_{Z[T_1 \cup T_2]} = \{ B \in \mathcal{B}_Z \mid B \subseteq T_1 \cup T_2 \}.
	\end{align}
	Therefore, \( D_{Z,\tau,\sigma} = ([n], \mathcal{F}) \) with
	\[
	\mathcal{F} = \{ \sigma(X \cap T_2) \mid X \in \mathcal{B}_{Z[T_1 \cup T_2]} \}
	= \{ \sigma(X \cap T_2) \mid X \in \mathcal{B}_Z, X \subseteq T_1 \cup T_2 \}.
	\]
	
	Conversely, we consider a vf-safe delta-matroid and construct a unique tight \( 3 \)-matroid
	via a given transversal \( 3 \)-tuple \( \tau \) and a projection \( \sigma \).
	
	Let \( D = ([n], \mathcal{F}) \) be a vf-safe delta-matroid. Fix a transversal \( 3 \)-tuple
	\( \tau = (T_1, T_2, T_3) \) of \( (U, \Omega) \) and a projection \( \sigma: U \rightarrow [n] \).
	We define
	\[
	\mathcal{B}_{D,\tau,\sigma} := \{ B \in \mathcal{T}(\Omega) \mid
	\sigma(B \cap T_2) \in \mathcal{F}(D \overline{\ast} \sigma(B \cap T_3)) \}.
	\]
	We denote by \( Z_{D,\tau,\sigma} = (U, \Omega, \mathcal{B}_{D,\tau,\sigma}^*) \) the tight
	\( 3 \)-matroid constructed by Brijder and Hoogeboom in \cite{Brijder2014}. We refer to this
	tight \( 3 \)-matroid as the \emph{lift} of \( D \) with respect to \( \tau \) and \( \sigma \).
	
	For an \( (n,3) \)-carrier \( (U, \Omega) \), a transversal \( 3 \)-tuple \( \tau = (T_1, T_2, T_3) \),
	and a projection \( \sigma: U \rightarrow [n] \), let \( i \in [n] \) and \( \omega = \sigma^{-1}(i) \).
	For \( k \in \{1,2,3\} \), let \( p_k \subseteq \omega \) be the skew pair of \( \omega \) with
	\( p_k \cap T_k = \emptyset \). Then we define:
	\begin{align}
		\tau \ast i &= (T_1 \triangle p_3, T_2 \triangle p_3, T_3) \\
		\tau + i &= (T_1, T_2 \triangle p_1, T_3 \triangle p_1) \\
		\tau \overline{\ast} i &= (T_1 \triangle p_2, T_2, T_3 \triangle p_2)
	\end{align}
	In particular, \( \tau \overline{\ast} i = ((\tau + i) \ast i) + i \). Since these operations commute
	on distinct elements of \( [n] \), we can write, for \( I \subseteq [n] \), e.g., \( \tau \ast I \)
	to denote the application of \( \ast i \) for all \( i \in I \) in arbitrary order.
	
	Let \( \mathbb{T}_{(U,\Omega)} \) denote the set of transversal \( 3 \)-tuples of \( (U, \Omega) \).
	Note that for any \( \tau, \tau' \in \mathbb{T}_{(U,\Omega)} \), there exists a sequence of operations
	\( \Gamma = g_1 1 g_2 2 \dots g_n n \), where \( g_i \in \{\mathbf{1}, \ast, +, \ast+, +\ast, \overline{\ast}\} \),
	such that \( \tau' = \tau \Gamma = \tau g_1 1 g_2 2 \dots g_n n \). Similarly, let \( \mathbb{P}_{(U,\Omega)} \) denote the set of projections of \( (U, \Omega) \).
	
	\begin{lemma}[\cite{Brijder2014}]\label{tuple}
		For a vf-safe delta-matroid \( D = ([n], \mathcal{F}) \), an \( (n,3) \)-carrier \( (U, \Omega) \),
		a transversal \( 3 \)-tuple \( \tau \in \mathbb{T}_{(U,\Omega)} \), a projection \( \sigma \in \mathbb{P}_{(U,\Omega)} \),
		and an element \( i \in [n] \), we have
		\[
		Z_{D,\tau,\sigma} = Z_{D+i,\tau+i,\sigma} = Z_{D \ast i,\tau \ast i,\sigma}.\]
	\end{lemma}
	
	For convenience, for a delta-matroid \( D \), we denote by \( \mathcal{F}(D) \) the feasible sets
	of \( D \) in the following context. Brijder and Hoogeboom showed in \cite{Brijder2014} that the mapping
	from \( D \) to \( Z_{D,\tau,\sigma} \) is the inverse of the mapping from \( Z \) to \( D_{Z,\tau,\sigma} \),
	which implies \( D = D_{Z_{D,\tau,\sigma},\tau,\sigma} \). For the sake of completeness, the following result
	elaborates on this correspondence in detail.
	
	\begin{lemma}[\cite{Brijder2014}]\label{inverse}
		For a vf-safe delta-matroid \( D \) with ground set \( [n] \), an \( (n,3) \)-carrier \( (U, \Omega) \),
		a transversal \( 3 \)-tuple \( \tau = (T_1, T_2, T_3) \), and a projection \( \sigma \), we have
		\[
		D = D_{Z_{D,\tau,\sigma},\tau,\sigma}.\]
	\end{lemma}
	
	\begin{proof}
		Denote the feasible sets of \( D_{Z_{D,\tau,\sigma},\tau,\sigma} \) by \( \mathcal{F}' \).
		Since the ground set of \( D_{Z_{D,\tau,\sigma},\tau,\sigma} \) is \( [n] \), it suffices to show that
		\( \mathcal{F}' = \mathcal{F}(D) \).
		
		The set of bases of \( Z_{D,\tau,\sigma} \) is
		\[
		\mathcal{B}_{D,\tau,\sigma} = \{ B \in \mathcal{T}(\Omega) \mid
		\sigma(B \cap T_2) \in \mathcal{F}(D \overline{\ast} \sigma(B \cap T_3)) \}.
		\]
		
		By equation (\ref{restrict}), the set of bases of \( Z_{D,\tau,\sigma}[T_1 \cup T_2] \) is
		\[
		\mathcal{B}_{Z_{D,\tau,\sigma}[T_1 \cup T_2]} = \{ B \in \mathcal{T}(\Omega) \mid
		B \subseteq T_1 \cup T_2 \text{ and } \sigma(B \cap T_2) \in \mathcal{F}(D) \}.
		\]
		
		Then \( \mathcal{F}' = \{ \sigma(B \cap T_2) \mid B \in \mathcal{B}_{Z_{D,\tau,\sigma}[T_1 \cup T_2]} \} \),
		and we have:
		\[
		\begin{aligned}
			F \in \mathcal{F}' &\Leftrightarrow F = \sigma(B \cap T_2) \text{ for some }
			B \in \mathcal{B}_{Z_{D,\tau,\sigma}[T_1 \cup T_2]} \\
			&\Leftrightarrow F = \sigma(B \cap T_2) \in \mathcal{F}(D) \text{ for some }
			B \in \mathcal{T}(\Omega) \text{ with } B \subseteq T_1 \cup T_2 \\
			&\Leftrightarrow F \in \mathcal{F}(D).
		\end{aligned}
		\]
	\end{proof}
	
	\begin{lemma}\label{trans}
		For a vf-safe delta-matroid \( D \) with ground set \( [n] \), an \( (n,3) \)-carrier \( (U,\Omega) \),
		a transversal \( 3 \)-tuple \( \tau = (T_1,T_2,T_3) \), and a projection \( \sigma \), we have
		\[
		Z_{D_\pi\Gamma(\hat{g}),\tau,\sigma} = Z_{D,\tau\Gamma(\hat{g}),\pi^{-1}\sigma}.
		\]
	\end{lemma}
	
	\begin{proof}
		By Lemma~\ref{tuple}, we have
		\[
		Z_{D\Gamma(\hat{g}),\tau,\sigma} = Z_{D,\tau\Gamma(\hat{g}),\sigma}.
		\]
		It suffices to show that for any transversal \( 3 \)-tuple \( \tau \) of \( (U,\Omega) \) and
		projection \( \sigma \) of \( (U,\Omega) \),
		\[
		Z_{D_\pi,\tau,\sigma} = Z_{D,\tau,\pi^{-1}\sigma}.
		\]
		By definition of \( \mathcal{B}_{D,\tau,\sigma} \), we have
		\[
		\mathcal{B}_{D_\pi,\tau,\sigma} = \{ B \in \mathcal{T}(\Omega) \mid
		\sigma(B \cap T_2) \in \mathcal{F}(D_\pi \overline{\ast} \sigma(B \cap T_3)) \},
		\]
		and
		\[
		\mathcal{B}_{D,\tau,\pi^{-1}\sigma} = \{ B \in \mathcal{T}(\Omega) \mid
		\pi^{-1}\sigma(B \cap T_2) \in \mathcal{F}(D \overline{\ast} \pi^{-1}\sigma(B \cap T_3)) \}.
		\]
		For any \( B \in \mathcal{T}(\Omega) \), let \( F_B = \sigma(B \cap T_2) \) and
		\( S_B = \sigma(B \cap T_3) \). Then:
		\[
		\begin{aligned}
			B \in \mathcal{B}_{D_\pi,\tau,\sigma}
			&\Leftrightarrow F_B \in \mathcal{F}(D_\pi \overline{\ast} S_B) \\
			&\Leftrightarrow \left| \{ X \in \mathcal{F}(D_\pi) \mid F_B \subseteq X \subseteq F_B \cup S_B \} \right| \text{ is odd} \\
			&\Leftrightarrow \left| \{ \pi(\pi^{-1}X) \in \mathcal{F}(D_\pi) \mid \pi(\pi^{-1}F_B) \subseteq \pi(\pi^{-1}X) \subseteq \pi(\pi^{-1}(F_B \cup S_B)) \} \right| \text{ is odd} \\
			&\Leftrightarrow \left| \{ \pi^{-1}X \in \mathcal{F}(D) \mid \pi^{-1}F_B \subseteq \pi^{-1}X \subseteq \pi^{-1}F_B \cup \pi^{-1}S_B \} \right| \text{ is odd} \\
			&\Leftrightarrow \pi^{-1}F_B \in \mathcal{F}(D \overline{\ast} \pi^{-1}S_B) \\
			&\Leftrightarrow \pi^{-1}\sigma(B \cap T_2) \in \mathcal{F}(D \overline{\ast} \pi^{-1}\sigma(B \cap T_3)) \\
			&\Leftrightarrow B \in \mathcal{B}_{D,\tau,\pi^{-1}\sigma}.
		\end{aligned}
		\]
		Thus, \( \mathcal{B}_{D_\pi,\tau,\sigma} = \mathcal{B}_{D,\tau,\pi^{-1}\sigma} \).
		Combining this with \( Z_{D\Gamma(\hat{g}),\tau,\sigma} = Z_{D,\tau\Gamma(\hat{g}),\sigma} \) yields
		\[
		Z_{D_\pi\Gamma(\hat{g}),\tau,\sigma} = Z_{D,\tau\Gamma(\hat{g}),\pi^{-1}\sigma}.
		\]
	\end{proof}
	
	The following theorem characterizes the elements of the orbit of a vf-safe delta-matroid \( D \):
	they are precisely those vf-safe delta-matroids that can be lifted to the same tight 3-matroid
	via some transversal tuple and projection.
	
	\begin{theorem}\label{orb}
		For a vf-safe delta-matroid \( D \) with ground set \( [n] \), an \( (n,3) \)-carrier \( (U,\Omega) \),
		a transversal \( 3 \)-tuple \( \tau = (T_1,T_2,T_3) \), and a projection \( \sigma \), we have
		\[
		\operatorname{Orb}(D) = \{ D' \mid \exists \tau' \in \mathbb{T}_{(U,\Omega)} \text{ and }
		\sigma' \in \mathbb{P}_{(U,\Omega)} \text{ such that } Z_{D',\tau',\sigma'} = Z_{D,\tau,\sigma} \}.
		\]
	\end{theorem}
	
	\begin{proof}
		We first prove the inclusion
		\[
		\operatorname{Orb}(D) \subseteq \{ D' \mid \exists \tau' \in \mathbb{T}_{(U,\Omega)} \text{ and }
		\sigma' \in \mathbb{P}_{(U,\Omega)} \text{ such that } Z_{D',\tau',\sigma'} = Z_{D,\tau,\sigma} \}.
		\]
		For any \( D' \in \operatorname{Orb}(D) \), there exists \( (\hat{g},\pi) \in \mathcal{G}^n \rtimes_\phi S_n \)
		such that \( D' = (\hat{g},\pi)D = D_\pi \Gamma(\hat{g}) \). Set \( \tau' = \tau \Gamma^{-1}(\hat{g}) \) and \( \sigma' = \pi \sigma \). Then, by Lemma~\ref{trans},
		\[
		\begin{aligned}
			Z_{D',\tau',\sigma'} &= Z_{D,\tau' \Gamma(\hat{g}), \pi^{-1} \sigma'} \\
			&= Z_{D,\tau, \sigma},
		\end{aligned}
		\]
		which implies \( D' \in \{ D' \mid \exists \tau' \in \mathbb{T}_{(U,\Omega)} \text{ and }
		\sigma' \in \mathbb{P}_{(U,\Omega)} \text{ such that } Z_{D',\tau',\sigma'} = Z_{D,\tau,\sigma} \} \).
		
		Conversely, for any \( D' \in \{ D' \mid \exists \tau' \in \mathbb{T}_{(U,\Omega)} \text{ and }
		\sigma' \in \mathbb{P}_{(U,\Omega)} \text{ such that } Z_{D',\tau',\sigma'} = Z_{D,\tau,\sigma} \} \),
		there exist \( \tau' \) and \( \sigma' \) such that \( Z_{D',\tau',\sigma'} = Z_{D,\tau,\sigma} \).
		
		There exist a sequence \( \Gamma(\hat{g}) \) and a permutation \( \pi \) such that
		\( \tau' = \tau \Gamma(\hat{g}) \) and \( \sigma' = \pi^{-1} \sigma \). Then, by Lemma~\ref{trans},
		\[
		\begin{aligned}
			Z_{D,\tau,\sigma} &= Z_{D',\tau',\sigma'} \\
			&= Z_{D',\tau \Gamma(\hat{g}), \pi^{-1} \sigma} \\
			&= Z_{D'_\pi \Gamma(\hat{g}), \tau, \sigma}.
		\end{aligned}
		\]
		Thus, by Lemma~\ref{inverse},
		\[
		D'_\pi \Gamma(\hat{g}) = D_{Z_{D'_\pi \Gamma(\hat{g}),\tau,\sigma},\tau,\sigma}
		= D_{Z_{D,\tau,\sigma},\tau,\sigma} = D,
		\]
		which implies \( D' \in \operatorname{Orb}(D) \).
	\end{proof}
	
	Through the proof of the above theorem, we can obtain the following corollary by analogy.
	
	\begin{corollary}
		For a vf-safe delta-matroid \( D \) with ground set \( [n] \), an \( (n,3) \)-carrier \( (U,\Omega) \),
		a transversal \( 3 \)-tuple \( \tau \), and a projection \( \sigma \), we have
		\[
		\operatorname{Orb}_\iota(D) = \{ D' \mid \exists \tau' \in \mathbb{T}_{(U,\Omega)} \text{ such that }
		Z_{D',\tau',\sigma} = Z_{D,\tau,\sigma} \}.
		\]
	\end{corollary}
	
	The following theorem characterizes the orbit of a vf-safe delta-matroid \( D \) as exactly
	all vf-safe delta-matroids obtained from a tight 3-matroid via different tuples and projections.
	
	\begin{theorem}\label{main}
		For a vf-safe delta-matroid \( D \) with ground set \( [n] \), an \( (n,3) \)-carrier \( (U,\Omega) \),
		a transversal \( 3 \)-tuple \( \tau = (T_1,T_2,T_3) \), and a projection \( \sigma \), we have
		\[
		\operatorname{Orb}(D) = \{ D_{Z_{D,\tau,\sigma},\tau',\sigma'} \mid
		\tau' \in \mathbb{T}_{(U,\Omega)},\ \sigma' \in \mathbb{P}_{(U,\Omega)} \}.
		\]
	\end{theorem}
	
	\begin{proof}
		For any \( D' \in \operatorname{Orb}(D) \), there exists \( (\hat{g},\pi) \in \mathcal{G}^n \rtimes_\phi S_n \)
		such that \( D' = (\hat{g},\pi)D = D_\pi \Gamma(\hat{g}) \). Set \( \tau' = \tau \Gamma(\hat{g})^{-1} \) and \( \sigma' = \pi \sigma \).
		By Lemma~\ref{trans}, we have
		\[
		Z_{D_\pi \Gamma(\hat{g}), \tau', \sigma'} = Z_{D, \tau' \Gamma(\hat{g}), \pi^{-1} \sigma'}.
		\]
		Then, by Lemma~\ref{inverse},
		\[
		\begin{aligned}
			D_\pi \Gamma(\hat{g}) &= D_{Z_{D_\pi \Gamma(\hat{g}), \tau', \sigma'}, \tau', \sigma'} \\
			&= D_{Z_{D, \tau' \Gamma(\hat{g}), \pi^{-1} \sigma'}, \tau', \sigma'} \\
			&= D_{Z_{D, \tau, \sigma}, \tau', \sigma'}.
		\end{aligned}
		\]
		Thus, \( D' \in \{ D_{Z_{D,\tau,\sigma},\tau',\sigma'} \mid \tau' \in \mathbb{T}_{(U,\Omega)},\ \sigma' \in \mathbb{P}_{(U,\Omega)} \} \).
		
		The reverse inclusion is analogous. For any \( D' \in \{ D_{Z_{D,\tau,\sigma},\tau',\sigma'} \mid \tau' \in \mathbb{T}_{(U,\Omega)},\ \sigma' \in \mathbb{P}_{(U,\Omega)} \} \),
		there exist \( \tau' \) and \( \sigma' \) such that \( D' = D_{Z_{D,\tau,\sigma},\tau',\sigma'} \).
		
		There exist a sequence \( \Gamma(\hat{g}) \) and a permutation \( \pi \) such that
		\( \tau = \tau' \Gamma(\hat{g}) \) and \( \sigma = \pi^{-1} \sigma' \). Then
		\[
		\begin{aligned}
			D' &= D_{Z_{D,\tau,\sigma},\tau',\sigma'} \\
			&= D_{Z_{D,\tau' \Gamma(\hat{g}), \pi^{-1} \sigma'},\tau',\sigma'} \\
			&= D_{Z_{D_\pi \Gamma(\hat{g}), \tau', \sigma'},\tau',\sigma'} \\
			&= D_\pi \Gamma(\hat{g}).
		\end{aligned}
		\]
		Thus, \( D' \in \operatorname{Orb}(D) \).
		
		Combining both inclusions gives
		\[
		\operatorname{Orb}(D) = \{ D_{Z_{D,\tau,\sigma},\tau',\sigma'} \mid \tau' \in \mathbb{T}_{(U,\Omega)},\ \sigma' \in \mathbb{P}_{(U,\Omega)} \}.
		\]
	\end{proof}
	
	\begin{corollary}
		For a vf-safe delta-matroid \( D \) with ground set \( [n] \), an \( (n,3) \)-carrier \( (U,\Omega) \),
		a transversal \( 3 \)-tuple \( \tau \), and a projection \( \sigma \), we have
		\[
		\operatorname{Orb}_\iota(D) = \{ D_{Z,\tau',\sigma} \mid Z = Z_{D,\tau,\sigma},\ \tau' \in \mathbb{T}_{(U,\Omega)} \}.
		\]
	\end{corollary}
	
	\section{Orbits of ribbon-graphic delta-matroids} \label{6}
	
	In this section, we investigate the orbit of ribbon-graphic delta-matroids using a degenerate form of Theorem~\ref{main}. A ribbon graph arises naturally from a classical cellularly embedded graph by taking a small neighborhood, and is formally defined as follows.
	
	\begin{definition}[\cite{Bollobas2002}]
		\normalfont
		A \emph{ribbon graph} \( G = (V(G), E(G)) \) is a surface with boundary, represented as the union of two sets of discs: a set \( V(G) \) of vertices and a set \( E(G) \) of edges, satisfying the following properties:
		\begin{enumerate}
			\item The vertices and edges intersect in disjoint line segments.
			\item Each such line segment lies on the boundary of exactly one vertex and exactly one edge.
			\item Every edge contains exactly two such line segments.
		\end{enumerate}
	\end{definition}
	
	Let \( G \) be a ribbon graph. A \emph{ribbon subgraph} of \( G \) is a ribbon graph obtained by removing vertices and edges from \( G \). A \emph{spanning ribbon subgraph} of \( G \) is a ribbon graph obtained by removing only edges from \( G \). A \emph{quasi-tree} \( Q \) is a connected ribbon graph with exactly one boundary component. If \( G \) is a connected ribbon graph, a \emph{spanning quasi-tree} \( Q \) of \( G \) is a spanning ribbon subgraph with exactly one boundary component. For disconnected graphs, we abuse notation by saying that \( Q \) is a spanning quasi-tree of \( G \) if \( k(Q) = k(G) \) (where \( k(G) \) denotes the number of connected components of \( G \)) and the connected components of \( Q \) are spanning quasi-trees of the connected components of \( G \).
	
	\begin{definition}[\cite{Chun2019}]
		\normalfont
		Let \( G = (V, E) \) be a ribbon graph with \( |E| = n \), and let \( o: E \rightarrow [n] \) be a bijection that induces an ordering on the edges. We define the set
		\[
		\mathcal{F}(G) := \left\{ F \subseteq [n] \mid o^{-1}(F) \text{ is the edge set of a spanning quasi-tree of } G \right\}.
		\]
		The pair \( D_o(G) := ([n], \mathcal{F}(G)) \) is called the \emph{delta-matroid of \( G \)}.
	\end{definition}
	
	A delta-matroid is said to be \emph{ribbon-graphic} if it is isomorphic to the delta-matroid of some ribbon graph. Note that Brijder and Hoogeboom $\cite{Brijder2013}$ showed that ribbon-graphic delta-matroids form a class of vf-safe delta-matroids.
	
	Let \( F = (V, E) \) be a 4-regular graph. A \emph{transition} \( \zeta_v \) at a vertex \( v \in V \) of \( F \) is a partition of the half-edges incident to \( v \) into two pairs. A \emph{transition system} \( T := \{ \zeta_v \mid v \in V \} \) of \( F \) is a choice of transition at each vertex of \( F \). Given a transition system \( T \) of \( F \), we obtain a 2-regular graph \( F|_T \) by \emph{splitting} each transition \( \zeta_v = \{b_1, b_2\} \) for \( v \in V \): specifically, we delete the vertex \( v \) and replace it with two new vertices \( v' \) and \( v'' \), where \( v' \) is incident to the two half-edges in \( b_1 \) and \( v'' \) is incident to the two half-edges in \( b_2 \).
	
	Since \( F \) is 4-regular, there are three distinct transitions at each vertex. Let \( U \) denote the set of all transitions of \( F \), and let \( \Omega \) be the partition of \( U \) where each skew class consists of the three transitions at the same vertex. Then \( (U, \Omega) \) forms a 3-carrier. A transition system \( T \) of \( F \) is precisely a transversal of \( \Omega \); recall that the set of all transversals of \( \Omega \) is denoted by \( \mathcal{T}(\Omega) \). Let
	\[
	\mathcal{B}(F) = \{ T \in \mathcal{T}(\Omega) \mid k(F|_T) = k(F) \},
	\]
	and define \( Z(F) := (U, \Omega, \mathcal{B}^*(F)) \). Bouchet showed in \cite{Bouchet2001} that the above construction yields a tight 3-matroid \( Z(F) \).
	
	We now construct a 4-regular graph, known as a \emph{medial graph}, starting from a ribbon graph. Let \( G \) be a ribbon graph. The \emph{medial graph} of \( G \), denoted by \( G_m \), is constructed by placing a vertex on each edge of \( G \) and then drawing the edges of \( G_m \) along the face boundaries of \( G \). For an isolated vertex, the medial graph is defined as a free loop, that is, a circular edge with no incident vertices.
	
	\begin{figure}
		\centering
		\includegraphics[width=0.9\linewidth]{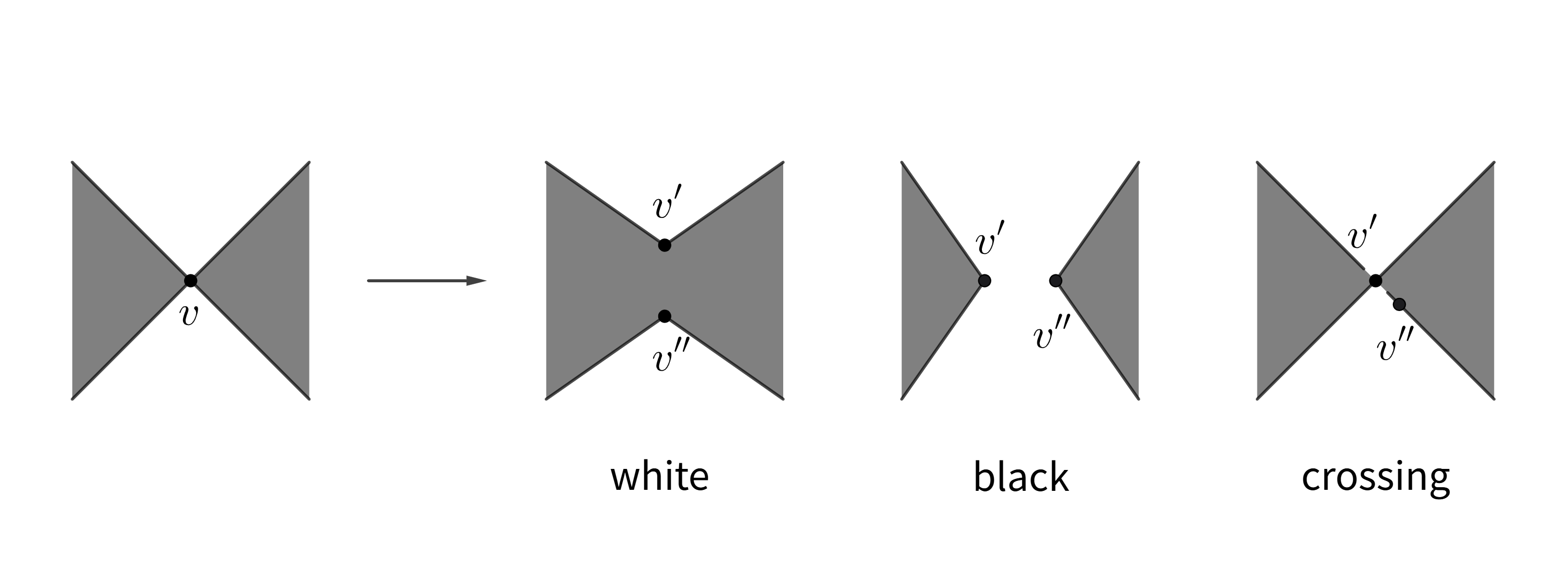}
		\caption{A vertex of \( G_m \) and its three vertex transitions}
		\label{fig:transition}
	\end{figure}
	
	Furthermore, any medial graph \( G_m \) is 4-regular and admits a \emph{checkerboard coloring}: color the faces of \( G_m \) that contain a vertex of the original graph \( G \) black, and the remaining faces white. We classify vertex transitions of \( G_m \) as follows: a transition is \emph{black} if it pairs half-edges sharing a black face, \emph{white} if it pairs half-edges sharing a white face, and \emph{crossing} otherwise, as illustrated in Figure~\ref{fig:transition}. Let \( T_b \) be the transition system consisting of all black transitions, \( T_w \) the transition system consisting of all white transitions, and \( T_c \) the transition system consisting of all crossing transitions.
	
	Let \( U \) denote the set of all transitions of \( G_m \), and let \( \Omega \) be the partition of \( U \) where each skew class consists of the three transitions associated with the same vertex of \( G_m \). Then \( (T_b, T_w, T_c) \) is a transversal 3-tuple of \( (U, \Omega) \), denoted by \( \tau_c \).
	
	\begin{theorem}
		Let \( G = (V, E) \) be a ribbon graph with \( |E| = n \), \( G_m \) its medial graph, and \( o: E \rightarrow [n] \) a bijection. Then
		\[
		\operatorname{Orb}(D_o(G)) = \left\{ D_{Z(G_m), \tau', \sigma'} \mid  \tau' \in \mathbb{T}_{(U,\Omega)},\ \sigma' \in \mathbb{P}_{(U,\Omega)} \right\}.
		\]
	\end{theorem}
	
	\begin{proof}
		Let \( \tau_c = (T_b, T_w, T_c) \) be the transversal 3-tuple defined as above, and let \( \sigma: U \rightarrow [n] \) be a projection of \( (U, \Omega) \) such that for any edge \( e \) in \( G \) and the corresponding vertex \( v_e \) in \( G_m \), \( o(e) = \sigma(\zeta_{v_e}) \). By Theorem~\ref{main}, it suffices to show that \[ Z(G_m) = Z_{D_o(G),\tau_c,\sigma}. \]
		
		For any transversal \( B \subseteq T_b \cup T_w \), the set \( B \cap T_w \) corresponds to transitions at some vertices \( V_B \) of \( G_m \), and these vertices \( V_B \) correspond to some edges \( E_B \) of \( G \). It can be seen that \( k(G_m|_B) = k(G_m) \) if and only if \( E_B \) is the edge set of a spanning quasi-tree of \( G \). Thus, we have
		\[
		\mathcal{B}_{Z(G_m)[T_b \cup T_w]} = \left\{ B \mid \sigma(B \cap T_w) = o(E_B) \in \mathcal{F}(D_o(G)) \right\} = \mathcal{B}_{Z_{D_o(G),\tau_c,\sigma}[T_b \cup T_w]}.
		\]
		
		Since \( Z(G_m) \) and \( Z_{D_o(G),\tau_c,\sigma} \) are tight, we prove by induction on \( |B \cap T_c| \) that
		\[
		\{ B \in \mathcal{B}_{Z(G_m)} \mid |B \cap T_c| = k \} = \{ B \in \mathcal{B}_{D_o(G),\tau_c,\sigma} \mid |B \cap T_c| = k \}.
		\]
		
		For the base case \( |B \cap T_c| = 0 \), we have \( \mathcal{B}_{Z(G_m)[T_b \cup T_w]} = \mathcal{B}_{Z_{D_o(G),\tau_c,\sigma}[T_b \cup T_w]} \). Assume the assertion holds for \( |B \cap T_c| = k-1 \). Let \( B \in \{ B \in \mathcal{B}_{Z(G_m)} \mid |B \cap T_c| = k \} \) and let \( x \in B \cap T_c \) belong to the skew class \( \omega \in \Omega \). Since \( Z(G_m) \) is tight, there is exactly one \( y \in \omega \) such that \( B \triangle \{x, y\} \notin \mathcal{B}_{Z(G_m)} \), while there exists a \( z \in \omega \) with \( y \neq z \) and \( z \neq x \) such that \( B \triangle \{x, z\} \in \mathcal{B}_{Z(G_m)} \). By the induction hypothesis,
		\[
		\{ B \in \mathcal{B}_{Z(G_m)} \mid |B \cap T_c| = k-1 \} = \{ B \in \mathcal{B}_{D_o(G),\tau_c,\sigma} \mid |B \cap T_c| = k-1 \}.
		\]
		Additionally, \( |(B \triangle \{x, y\}) \cap T_c| = k-1 \) and \( |(B \triangle \{x, z\}) \cap T_c| = k-1 \). It therefore follows that \( B \triangle \{x, y\} \notin \mathcal{B}_{D_o(G),\tau_c,\sigma} \) and \( B \triangle \{x, z\} \in \mathcal{B}_{D_o(G),\tau_c,\sigma} \). Since \( Z_{D_o(G),\tau_c,\sigma} \) is tight, it follows that \( B \in \mathcal{B}_{D_o(G),\tau_c,\sigma} \). Therefore,
		\[
		\{ B \in \mathcal{B}_{Z(G_m)} \mid |B \cap T_c| = k \} \subseteq \{ B \in \mathcal{B}_{D_o(G),\tau_c,\sigma} \mid |B \cap T_c| = k \},
		\]
		and the reverse inclusion follows similarly.
		
		Hence, \( Z(G_m) = Z_{D_o(G),\tau_c,\sigma} \), and the conclusion holds.
	\end{proof}
	
	\section*{Acknowledgements}
	This work is supported by NSFC (Nos. 12571379, 12471326).
	
	
\end{document}